\documentclass{amsart}

\newtheorem{theorem}{Theorem}%[section]
\newtheorem{lemma}[theorem]{Lemma}

\newtheorem{corollary}[theorem]{Corollary}

\theoremstyle{definition}

\newtheorem{example}[theorem]{Example}
\newtheorem{remark}[theorem]{Remark}

\usepackage{amscd,amssymb}
\begin{document}

\title[Central polynomials]
{Central polynomials for matrices\\
over finite fields}

\author[Matej Bre\v{s}ar and Vesselin Drensky]
{Matej Bre\v{s}ar and Vesselin Drensky}
\address{Faculty of Mathematics and Physics, University of Ljubljana, and Faculty of Natural Sciences and Mathematics,
University of Maribor, Slovenia}
\email{matej.bresar@fmf.uni-lj.si}
\address{Institute of Mathematics and Informatics,
Bulgarian Academy of Sciences,
1113 Sofia, Bulgaria}
\email{drensky@math.bas.bg}
\thanks
{The research of the first named author was partially
supported by the Slovenian Research Agency (Program No. P1-0288)}

\subjclass[2010]
{16R10, 16R30.}
\keywords{Central polynomials, matrix algebras.}

\begin{abstract} Let $c(x_1,\ldots,x_d)$ be a multihomogeneous central polynomial for the $n\times n$ matrix algebra
$M_n(K)$ over an infinite field $K$ of positive characteristic $p$. We show that there exists a multihomogeneous polynomial
$c_0(x_1,\ldots,x_d)$ of the same degree and with coefficients in the prime field ${\mathbb F}_p$ which is central for
the algebra $M_n(F)$ for any (possibly finite) field $F$ of characteristic $p$. The proof is elementary and uses standard combinatorial
techniques only.
\end{abstract}

\maketitle

\centerline{\it Dedicated to Ed Formanek on the occasion of his 70th birthday}

\section{Introduction}

Let $K$ be a field. The polynomial $c(x_1,\ldots,x_d)$ in the free associative algebra
$K\langle X\rangle=K\langle x_1,x_2,\ldots\rangle$ is called a central polynomial
for the $K$-algebra $R$ if $c(x_1,\ldots,x_d)$, when evaluated on $R$, belongs to the center of $R$,
i.e.,
\[
[c(x_1,\ldots,x_d),x_{d+1}]:=c(x_1,\ldots,x_d)x_{d+1}-x_{d+1}c(x_1,\ldots,x_d)=0
\]
is a polynomial identity for $R$
and $c(x_1,\ldots,x_d)=0$ is not.
Usually one requires that $c(x_1,\ldots,x_d)$
is multihomogeneous, i.e., homogeneous in each variable $x_1,\ldots,x_d$, in order to avoid the trivial case
$c=\text{constant}+\text{polynomial identity}$.
This is not a restriction when $K$ is infinite because standard Vandermonde arguments give that every polynomial
identity is equivalent to a set of multihomogeneous identities.
Answering the famous problem of Kaplansky \cite{K1, K2}, in 1972-73 Formanek \cite{F} and Razmyslov \cite{R}
independently constructed central polynomials for the $n\times n$ matrix algebras $M_n(K)$, for any size $n$ and any field $K$.
This led to a significant revision of structure theory of algebras with polynomial identity.
The construction of Formanek is based on generic matrices and the fact that the eigenvalues of
a generic matrix are pairwise different.
For a finite field $K$ the proof involves an additional argument that $M_n(K)$ contains a matrix with pairwise different eigenvalues,
which is obvious for $K$ infinite.
The central polynomial of Razmyslov is multilinear. Hence it is sufficient to show that it has scalar values
when evaluated on the matrix units only, which does not depend on the cardinality of $K$.
Before, in 1969, Latyshev and Shmelkin \cite{LS} constructed a central (clearly nonhomogeneous) polynomial in one variable
for the matrix algebra $M_n(K)$ over a finite field $K$.
See, e.g., the book of Drensky and Formanek \cite{DF}
for the importance of central polynomials for the theory and for different proofs of their existence, as well as
for a background on PI-algebras.

Kharchenko \cite{Kh} in 1979, Braun \cite{Br} in 1982, and recently Bre\v{s}ar \cite{B1, B2} in 2011-12,
gave nonconstructive proofs of the existence of central polynomials when the base field $K$ is infinite.
The main ideas of the proofs are similar. But Kharchenko and Braun used old results by Amitsur from the 1950's,
before Kaplansky stated his problem,
while Bre\v{s}ar's proof is more self-contained and uses techniques typical for
 generalized polynomial identities and functional identities.

The purpose of this note is to show that if $K$ is an infinite field of positive characteristic $p$
and $c(x_1,\ldots,x_d)$ is a multihomogeneous central polynomial for $M_n(K)$,
then there exists a multihomogeneous polynomial
$c_0(x_1,\ldots,x_d)$ of the same degree and with coefficients in the prime field ${\mathbb F}_p$ which is central for
the algebra $M_n(F)$ for any (possibly finite) field $F$ of characteristic $p$. The proof is elementary and uses only standard combinatorial
techniques. It also gives  an algorithm how to find $c_0$ if we know $c$.
In this way we complete the nonconstructive proofs in \cite{Kh}, \cite{Br},
and \cite{B1, B2},
removing the requirement of the infinity of the field $K$.

\section{The main result}

We fix an infinite field $K$ of any characteristic. If $f(x_1,\ldots,x_d)\in K\langle X\rangle$
is a multihomogeneous polynomial and $\deg_{x_1}f=m$, we consider
the polynomial
\[
f(y_{11}+\cdots+y_{1m},x_2,\ldots,x_d)=\sum_a f_{1,a}(y_{11},\ldots,y_{1m},x_2,\ldots,x_d)
\]
in $K\langle X,y_{11},\ldots,y_{1m}\rangle$,
where $a=(a_1,\ldots,a_m)$ and the polynomial $f_{1,a}$ is homogeneous of degree
$a_i\geq 0$ in $y_{1i}$, $i=1,\ldots,m$.
If $a=b^{(1)}=(b_{11},\ldots,b_{1k_1},0,\ldots,0)$, where $b_{11},\ldots,b_{1k_1}>0$, we call the polynomial $f_{1,a}$
a partial linearization of $f$ with respect to $x_1$ and use the notation
\[
f_{1,a}=f(b^{(1)}:y_{11},\ldots,y_{1k_1}\vert x_2,\ldots,x_d).
\]
Continuing this process with the other variables $x_2,\ldots,x_d$, we obtain all partial linearizations
\[
f(b:y_{ij})=f(b^{(1)}:y_{11},\ldots,y_{1k_1}\vert\ldots\vert b^{(d)}:y_{d1},\ldots,y_{dk_d})
\]
of $f$.
The main tool of our approach is an elementary lemma. It mimics the following facts:
\begin{enumerate}
\item[$\bullet$]
In characteristic 0 every polynomial identity is equivalent to a set of multilinear polynomial identities.
\item[$\bullet$]
In order to check that a multilinear polynomial is an identity for an algebra it is sufficient to evaluate it on a basis of the algebra.
\end{enumerate}
If $K$ is of positive characteristic $p$,
there is a much stronger version of the lemma which states that it is sufficient to consider only partial linearizations $f(b:y_{ij})$
such that all degrees $\deg_{y_{ij}}f(b:y_{ij})$ are powers of the characteristic $p$. The proof is given in
Drensky \cite[Theorem 1.4]{D} for Lie algebras, see also \cite[Theorem 4.2.6, p. 144 of the Russian original]{Ba},
with the same arguments working for associative algebras.
To make transparent the technical notation in the proof of the lemma,
we give a small example.

\begin{example}
Let $f(x_1)=x_1^3$. Then $f(y_1+y_2+y_3)$ is equal to
\begin{eqnarray*}
&&y_1^3+y_2^3+y_3^3\\
&+&(y_1^2y_2+y_1y_2y_1+y_2y_1^2)+(y_1^2y_3+y_1y_3y_1+y_3y_1^2)
+(y_2^2y_3+y_2y_3y_2+y_3y_2^2) \\
&+&(y_1y_2^2+y_2y_1y_2+y_2^2y_1) +
(y_1y_3^2+y_3y_1y_3+y_3^2y_1)+(y_2y_3^2+y_3y_2y_3+y_3^2y_2)\\
&+&(y_1y_2y_3+y_1y_3y_2+y_2y_1y_3+y_2y_3y_1+y_3y_1y_2+y_3y_2y_1).
\end{eqnarray*}
The partial linearizations are
\begin{eqnarray*}
f((3):y_1)&=&y_1^3,\\
f((2,1):y_1,y_2)&=&y_1^2y_2+y_1y_2y_1+y_2y_1^2,\\
f((1,2):y_1,y_2)&=&y_1y_2^2+y_2y_1y_2+y_2^2y_1,\\
f((1,1,1):y_1,y_2,y_3)&=&y_1y_2y_3+y_1y_3y_2+y_2y_1y_3+y_2y_3y_1+y_3y_1y_2+y_3y_2y_1.%\\
\end{eqnarray*}
If $R$ is an algebra with basis $V=\{v_j\mid j=1,2,\ldots\}$,
then $r\in R$ can be written as
$r=\xi_1v_1+\cdots+\xi_nv_n$, where $\xi_j\in K$.
In the case of three summands, if
$r=\xi_1v_1+\xi_2v_2+\xi_3v_3$, then
\begin{eqnarray*}
f(r)&=&f(\xi_1v_1+\xi_2v_2+\xi_3v_3)\\
&=&\xi_1^3f((3):v_1)+\xi_2^3f((3):v_2)+\xi_3^3f((3):v_3)\\
&+&\xi_1^2\xi_2f((2,1):v_1,v_2)+\xi_1^2\xi_3f((2,1):v_1,v_3)
+\xi_2^2\xi_3f((2,1):v_2,v_3)\\
&+&\xi_1\xi_2^2f((1,2):v_1,v_2)+\xi_1\xi_3^2f((1,2):v_1,v_3)
+\xi_2\xi_3^2f((1,2):v_2,v_3)\\
&+&\xi_1\xi_2\xi_3f((1,1,1):v_1,v_2,v_3).%\\
\end{eqnarray*}
In the general case, if $r=\sum\xi_jv_j$, then
\[
f(r)=\sum_{(b_1,\ldots,b_k)}\sum_{j_1<\ldots<j_k}\xi_{j_1}^{b_1}\cdots\xi_{j_k}^{b_k}f((b_1,\ldots,b_k):v_{j_1},\ldots,v_{j_k}).
\]
Hence the evaluations of $f$ on $R$ are linear combinations of evaluations of the partial linearizations of $f$ on the basis of $R$.
\end{example}

\begin{lemma}\label{equivalent system of PIs}
Let $f(x_1,\ldots,x_d)\in K\langle X\rangle$ be a multihomogeneous polynomial, the field $K$ being infinite.
Let $R$ be a PI-algebra with basis $V=\{v_{\alpha}\mid \alpha\in J\}$ as a $K$-vector space. Then $f=0$ is a polynomial identity
for $R$ if and only if
$f(b:v_{\alpha_{ij}})=0$ for all partial linearizations $f(b: y_{ij})$ and all $\alpha_{ij}\in J$.
\end{lemma}

\begin{proof}
Since the field $K$ is infinite, the identities $f(b: y_{ij})=0$ are consequences of the polynomial identity $f=0$.
If $f=0$ is an identity for $R$, then all $f(b: y_{ij})$ vanish on $R$ and hence on its basis $V$. Let us assume that
$f=0$ is not a polynomial identity for $R$ and $f(r_1,\ldots,r_d)\not=0$ for some $r_1,\ldots,r_d\in R$. We write the elements
$r_i$ as finite sums of the form
\[
r_i=\xi_{i1}v_{\alpha_{i1}}+\cdots+\xi_{in_i}v_{\alpha_{in_i}},\quad \xi_{ij}\in K,\alpha_{ij}\in J.
\]
Then $f(r_1,\ldots,r_d)$ can be expressed as a linear combination of evaluations of the partial linearizations on the basis $V$ of $R$ as
\[
\sum_b\sum\prod_{i=1}^d\prod_{j=1}^{k_i}\xi_{iq_{ij}}^{b_{ij}}
f(b^{(1)}:v_{\alpha_{1q_{11}}},\ldots,v_{\alpha_{1q_{1k_1}}}\vert\ldots\vert b^{(d)}:v_{\alpha_{dq_{d1}}},\ldots,y_{\alpha_{dq_{dk_d}}}),
\]
where the outer sum runs over all partial linearizations of $f$ and the inner sum is over all
$q_{11}<\cdots<q_{1k_1},\ldots,q_{d1}<\cdots<q_{dk_d}$. Since $f(r_1,\ldots,r_d)\not=0$ we have that some
$f(b:v_{\alpha_{iq_j}})$ is also different from 0.
\end{proof}

Let $S$ be a subfield of the field $K$ and let $A$ be an $S$-algebra with basis $V=\{v_{\alpha}\mid \alpha\in J\}$. Then the tensor product
$_KA=K\otimes_SA$ is a $K$-algebra with basis $1\otimes V=\{1\otimes v_{\alpha}\mid \alpha\in J\}$. We shall identify the elements
$1\otimes v_{\alpha}$ and $v_{\alpha}$, and shall consider $A$ as an $S$-subalgebra of $_KA$.
Similarly, we shall consider the free $S$-algebra $S\langle X\rangle$ as an $S$-subalgebra of the free $K$-algebra $K\langle X\rangle$.
Hence every nonzero polynomial $f\in K\langle X\rangle$ has the form
$f=\eta_1f_1+\cdots+\eta_qf_q$, where $f_1,\ldots,f_q\in S\langle X\rangle$
and the elements $\eta_1,\ldots,\eta_q$ of $K$ are linearly independent over $S$.

\begin{theorem}\label{main theorem}
Let $S$ be a subfield of an infinite field $K$ and let $A$ be an $S$-algebra. Let
$f=\eta_1f_1+\cdots+\eta_qf_q$ be a  multihomogeneous polynomial in $K\langle X\rangle$,
where $f_1,\ldots,f_q\in S\langle X\rangle$
and $\eta_1,\ldots,\eta_q\in K$ are linearly independent over $S$.

{\rm (i)} If $f=0$ is a polynomial identity for the $K$-algebra $_KA$, then $f_t=0$ are polynomial identities for
the $F$-algebra $_FA$ for all $t=1,\ldots,q$ and all extensions $F$ of the field $S$.

{\rm (ii)} If $f$ is a central polynomial for $_KA$, then at least one $f_i$ is a central polynomial for
$_FA$ for all extensions $F$ of $S$.
\end{theorem}

\begin{proof} Let $f(b:y_{ij})$ and $f_t(b:y_{ij})$ be the partial linearizations of $f$ and $f_t$, $t=1,\ldots,q$, respectively.
It is clear from the process of partial linearization that
\[
f(b:y_{ij})=\eta_1f_1(b:y_{ij})+\cdots+\eta_qf_q(b:y_{ij}).
\]
The evaluations of $f(b:y_{ij})$ on a basis $V=\{v_{\alpha}\mid \alpha\in J\}$ of $A$ give
\[
f(b:v_{\alpha_{ij}})=\eta_1f_1(b:v_{\alpha_{ij}})+\cdots+\eta_qf_q(b:v_{\alpha_{ij}}).
\]
The evaluation $f_t(b:v_{\alpha_{ij}})$ of $f_t(b:y_{ij})\in S\langle y_{ij}\rangle$ belongs to $A$
and $\eta_1,\ldots,\eta_q\in K$ are linearly independent over $S$. Hence, if $f(b:v_{\alpha_{ij}})=0$, then
$f_t(b:v_{\alpha_{ij}})=0$ for all $t=1,\ldots,q$.

(i) If $f=0$ is a polynomial identity for $_KA$, then $f(b:y_{ij})$ vanishes evaluated on the basis $V$ of $A$
because $A\subset{}_KA$. Hence the evaluations of $f_t(b:y_{ij})$ also vanish on $V$ for all $t=1,\ldots,q$.
If $F$ is an extension of $S$, then the evaluations of $f_t$ on $_FA$ are linear combinations of
evaluations of $f_t(b:y_{ij})$ on the basis $V$ of $A$ which is a basis also of $_FA$.
Applying Lemma \ref{equivalent system of PIs} we obtain that $f_t=0$
is a polynomial identity for $_FA$.

(ii) If $f=f(x_1,\ldots,x_d)$ is a central polynomial for $_KA$, then the commutator
$[f,x_{d+1}]$ is a polynomial identity for $_KA$. Since
\[
[f,x_{d+1}]=\eta_1[f_1,x_{d+1}]+\cdots+\eta_q[f_q,x_{d+1}],
\]
part (i) of the theorem gives that $[f_t,x_{d+1}]=0$
is a polynomial identity for $_FA$ for any $t=1,\ldots,q$.
This means that
$f_t$ is either a central polynomial or a polynomial identity for $_FA$. Since $f=0$ is not a polynomial identity for $_KA$,
by Lemma \ref{equivalent system of PIs}
there exists a nonzero evaluation $f(b:v_{\alpha_{ij}})$. Hence $f_t(b:v_{\alpha_{ij}})\not=0$ for some $t$ and $f_t$ is not a polynomial identity.
Hence it is a central polynomial for $_FA$ for all extensions $F$ of $S$.
\end{proof}

\begin{corollary}\label{central polynomials}
Let $c(x_1,\ldots,x_d)$ be a multihomogeneous central polynomial for
$M_n(K)$ over an infinite field $K$ of positive characteristic $p$. There exists a multihomogeneous polynomial
$c_0(x_1,\ldots,x_d)$ of the same degree and with coefficients in the prime field ${\mathbb F}_p$ which is central for
the algebra $M_n(F)$ for any field $F$ of characteristic $p$.
\end{corollary}

\begin{proof}
Let $c=\eta_1c_1+\cdots+\eta_qc_q$, where $c_t\in{\mathbb F}_p\langle X\rangle$
and $\eta_1,\ldots,\eta_q\in K$ are linearly independent over ${\mathbb F}_p$.
Since $M_n(K)={}_KM_n({\mathbb F}_p)$, Theorem \ref{main theorem}
gives that at least one of the polynomials $c_t$ is central for the algebra $M_n(F)$.
Clearly, $c_t=c_t(x_1,\ldots,x_d)$ is multihomogeneous of the same multidegree as $c(x_1,\ldots,x_d)$.
\end{proof}

\begin{remark}
If the field $K$ is constructive, the proof of Theorem \ref{main theorem} gives an algorithm how to construct the central polynomial
$c_0\in {\mathbb F}_p\langle X\rangle$ if we know $c\in K\langle X\rangle$. We write $c$ in the form
$c=\eta_1c_1+\cdots+\eta_qc_q$ with $\eta_1,\ldots,\eta_q\in K$ linearly independent over ${\mathbb F}_p$, form all
partial linearizations $c_t(b:y_{ij})$ and evaluate them on the matrix units $e_{q_iq_j}$ of $M_n({\mathbb F}_p)$.
There exists a $c_t(b:e_{q_iq_j})$ which is not zero. Then the corresponding $c_t$ is a central polynomial.
\end{remark}

\end{document}